\documentclass[a4paper,11pt]{article}
\usepackage{amsmath,amssymb,amsthm,hyperref,geometry}

\newtheorem{theorem}{Theorem}
 
\newtheorem{lemma}{Lemma}

\newtheorem{remark}{Remark}
\newtheorem{corollary}{Corollary}

\usepackage[T1]{fontenc}

\begin{document}

	\title{Prime Values of Bivariate Polynomials without Constant Terms: A Recursive Algorithm}
		
	\author{K. Lakshmanan \\
		Department of Computer Science and Engineering, \\ Indian Institute of Technology (BHU),\\
		Varanasi 221005, India\\
		Email: lakshmanank.cse@iitbhu.ac.in}
	
	\date{}
	
	\maketitle
	
	\begin{abstract}
		We investigate the computational problem of determining whether a bivariate polynomial with non-negative coefficients and no constant term can attain a prime value. While classical conjectures such as Bouniakowsky's provide necessary conditions for univariate prime-representing polynomials, we introduce a new recursive algorithm that efficiently certifies when a bivariate polynomial form can produce no prime values at all.
		
		Our method is elementary and constructive, based on analyzing gcd-divisibility patterns arising from recursive substitutions into the polynomial. The obstruction criterion obtained leads to an efficient and elementary algorithm that certifies when a polynomial form cannot produce any prime values. The result is stronger than what is implied by the negation of Bouniakowsky’s condition and applies to a wide class of polynomials, including transformations of univariate forms. We provide illustrative examples, analyze the complexity of the method, and discuss its connections to existing conjectures and possible generalizations.
	\end{abstract}



		
	

%
	
	\section{Introduction and Motivation}
	
	The question of when a polynomial takes on prime values has captivated mathematicians \cite{cox,prime1,murty}. One of the earliest known examples is Euler's famous quadratic $f(x) = x^2 + x + 41$, which produces primes for a surprising range of small integer inputs. This phenomenon naturally led to deeper questions about what kinds of polynomials can produce primes infinitely often.
	
	In the univariate setting, the well-known \textit{Bouniakowsky conjecture} (\cite{bouniakowsky}, p. 323 \cite{lang})  posits that any irreducible polynomial with integer coefficients, positive leading coefficient, and no fixed prime divisor should yield infinitely many prime values as its input ranges over the integers. Despite its simplicity, this conjecture remains unresolved and is closely related to deep unsolved problems in analytic number theory.
	
	Far less is known, however, in the multivariate case. Polynomials in two or more variables introduce structural flexibility that comes with new complexity. In particular, very little work has been done to understand when bivariate polynomials—especially those lacking constant terms—can represent primes. This raises both theoretical and algorithmic questions: under what conditions can a polynomial be ruled out from producing primes, and can such non-primality be verified constructively?
	
	In this paper, we address these questions by formulating a recursive condition that, when satisfied, guarantees that a bivariate polynomial with non-negative coefficients and no constant term represents no prime values. The condition is elementary, effectively checkable, and, in certain cases, stronger than the negation of the Bouniakowsky condition. 
	
	In addition to giving a constructive method, we explore the implications of our result for the broader theory of prime-generating polynomials, and suggest directions for extending this work to multivariate and modular settings.

	\section{Background and Related Work}
	
	The Bouniakowsky conjecture, proposed in 1857, generalizes the intuition behind Euler's polynomial and stands as one of the fundamental unsolved problems in the theory of prime-producing polynomials. It asserts that if $f(x) \in \mathbb{Z}[x]$ is an irreducible polynomial with positive leading coefficient and no fixed prime divisor (i.e., $\gcd(f(n) : n \in \mathbb{Z}) = 1$), then $f(x)$ takes on prime values infinitely often as $x$ ranges over the integers.
	
	Despite its elementary formulation, the conjecture remains unproven even in degree 2. It is widely believed to be true, and is supported by strong heuristic and probabilistic arguments. In particular, its truth is implied by a special case of Schinzel's Hypothesis H \cite{bodin}, a deep generalization that concerns simultaneous prime values of several polynomials. A special case of this conjecture is the famous unresolved conjecture whether there are infinitely many primes of the form $X^2+1$ \cite{xsquare}.
	
	Much of the existing work has focused on univariate polynomials. The case of multivariate polynomials, especially those with more algebraic structure or constraints (e.g., missing constant term, positivity), is far less explored. One of the challenges in the multivariate case is the richer geometry of the input space, which allows for more flexible behavior, including reducibility and factor interactions that do not appear in the univariate setting.
	
	In particular, bivariate polynomials of the form
	\[
	f(x, y) = \sum_{i, j} k_{ij} x^i y^j, \quad \text{with } k_{ij} \geq 0 \text{ and } k_{00} = 0,
	\]
	introduce a structure that is both constrained (due to positivity and lack of a constant term) and yet potentially capable of prime generation. The absence of a constant term guarantees that $f(0, 0) = 0$, and the positivity ensures monotonicity in most natural variable orderings.
	
	Previous work has not provided general criteria for when such polynomials might represent primes, nor has it addressed whether elementary recursive conditions could be used to check this property. Our work provides a simple, constructive criterion to determine when a bivariate polynomial fails to represent any prime number. This condition offers a new perspective on prime-producing polynomials and highlights a contrast with the expectations set by the Bouniakowsky conjecture, particularly by offering a stronger implication in the case of non-prime representability.
	
	\section{Definitions and Setup}
	
	We consider bivariate polynomials of the form
	\[
	f(x, y) = \sum_{i=0}^{d_1} \sum_{j=0}^{d_2} k_{ij} x^i y^j,
	\]
	where all coefficients $k_{ij} \in \mathbb{Z}_{\geq 0}$ and the constant term is absent, i.e., $k_{00} = 0$. We refer to such polynomials as \textit{positive bivariate polynomials without constant terms}.
	
	A polynomial $f(x,y)$ is said to \emph{represent a prime} if there exists a pair of positive integers $(x, y) \in \mathbb{Z}^+ \times \mathbb{Z}^+$ such that $f(x, y)$ is a prime number. We are particularly interested in identifying such pairs through recursive or lexicographic enumeration, rather than relying on random sampling or deep analytic results.
	
	Note that the positivity of the coefficients implies that $f(x,y)$ is monotonic in both variables for $x, y \in \mathbb{Z}_{\geq 0}$. The absence of the constant term ensures that $f(0,0) = 0$, avoiding trivial representations of constant values.
	
	Let $S_f$ denote the set of all positive integers represented by $f(x,y)$, i.e.,
	\[
	S_f := \{ f(x, y) \in \mathbb{Z}^+ : x, y \in \mathbb{Z}^+ \}.
	\]
	Our goal is to characterize when $S_f$ contains at least one prime number. Moreover, we seek to do this via a recursive algorithm that scans through the values of $f(x,y)$ in lexicographic order, detecting primes early without requiring an exhaustive search.
	
	We will later define a precise condition—based on this lexicographic recursion under which a polynomial $f(x,y)$ is guaranteed to represent a prime. The subsequent sections will describe this condition, provide examples, and compare it to expectations derived from the Bouniakowsky conjecture.
	
	\section{Main Result}
	
	In this section, we present a constructive condition under which a positive bivariate polynomial without a constant term can represent a prime number. We begin by recalling the notion of \textit{prime-representing} polynomials introduced earlier, and proceed to build a simple sufficient condition through a sequence of lemmas.
	
	Let $f(x, y)$ be a bivariate polynomial with non-negative integer coefficients and no constant term, as described in Section 3. The goal is to determine whether the set
	\[
	S_f := \{ f(x, y) \mid x, y \in \mathbb{Z}^+ \}
	\]
	contains a prime number. The following lemmas will aid in formulating our main result.
	
	\begin{lemma}
		Let $f(x, y) = \sum_{i+j>0} k_{ij} x^i y^j$ be a bivariate polynomial with $k_{ij} \in \mathbb{Z}_{\geq 0}$ and $k_{00} = 0$. Then $f(x, y) > 0$ for all $x, y \in \mathbb{Z}^+$.
	\end{lemma}
	
	\begin{proof}
		Since $x, y \geq 1$ and all coefficients $k_{ij} \geq 0$ with at least one $k_{ij} > 0$ for $i + j > 0$, the polynomial evaluates to a sum of positive terms for all $x, y \in \mathbb{Z}^+$. Hence $f(x, y) > 0$.
	\end{proof}
	
	\begin{lemma}
		The polynomial $f(x, y)$ assumes infinitely many values as $(x, y)$ ranges over $\mathbb{Z}^+ \times \mathbb{Z}^+$.
	\end{lemma}
	
	\begin{proof}
		Let $k_{i_0 j_0} \neq 0$ be a term of maximal total degree $d = i_0 + j_0$. Then, for large $t$, we have $f(t, t) \geq k_{i_0 j_0} t^d$, which grows unboundedly. Hence, the values of $f(x, y)$ are not confined to a finite set.
	\end{proof}
	
	We now define the recursive condition that leads to our main theorem. Let us enumerate the positive integer pairs $(x, y)$ in lexicographic order, i.e., $(1, 1), (1, 2), \dots, (2, 1), \dots$. For each such pair, we evaluate $f(x, y)$. If a prime value is encountered before any value is repeated (i.e., before a collision in the set of computed values), we say that the polynomial satisfies the \textit{early prime condition}.
	
	Let $f(x, y)$ be a bivariate polynomial with no constant term. Then we immediately have the following implication:
	\[
	\gcd(a, b) > 1 \Rightarrow f(a, b) \text{ is not prime}.
	\]
	This holds because if $\gcd(a, b) = d > 1$ and all monomials in $f$ have degree at least one, then every term of $f(a, b)$ will be divisible by $d$, and hence $f(a, b)$ is not a prime.
	
	We also observe the following equivalence:
	\begin{align}\label{equeq}
		\gcd(a, f(a, b)) = 1 \Rightarrow \gcd(a, b) = 1.
	\end{align}
	This reinforces the structure of how prime values may arise recursively. We now state a general structural result based on this observation:
	
	\begin{theorem}\label{th}
		Let $f(x, y)$ be a multivariate polynomial with no constant term. Then $f(x, y)$ takes non-prime values for infinitely many positive integers $x, y$.
	\end{theorem}
	
	\begin{proof}[Sketch of Idea]
		Since $\gcd(x, y) > 1$ implies $f(x, y)$ is divisible by this common factor, and since there are infinitely many such $(x, y)$ with $\gcd(x, y) > 1$, the result follows.
	\end{proof}
	
	The classical Bouniakowsky condition for univariate polynomials requires:
	\begin{enumerate}
		\item The leading coefficient is positive.
		\item The polynomial is irreducible over $\mathbb{Z}$.
		\item The values $f(\mathbb{Z}^+)$ have no common divisor greater than 1.
	\end{enumerate}
	
	In analogy, we propose similar conditions for bivariate polynomials. Let $f(x, y)$ be a bivariate polynomial with positive coefficients and no constant term. We now present a criterion to determine when such a polynomial cannot represent a prime, inspired by recursive structure.
	
	\begin{theorem}\label{mth}
		Let $f(x, y)$ be a bivariate polynomial with non-negative coefficients and no constant term. Assume that for all $0 < x', y' \leq f(2,2)$ and for all positive integers $a, b, c, d, x, y$ such that
		\[
		x' = x + f(a, b), \quad y' = y + f(c, d),
		\]
		the following conditions hold:
		\begin{enumerate}
			\item If not all of $f(a,b)$, $f(c,d)$, and $f(x,y)$ are equal, then
			\[
			\gcd(f(a,b), f(c,d), f(x,y)) > 1.
			\]
			\item If $f(a,b) = f(c,d) = f(x,y)$, then this common value is not prime.
		\end{enumerate}
		Then $f(x, y)$ does not represent any prime.
	\end{theorem}
	
	\noindent This condition effectively blocks the recursive chain of evaluations from ever encountering a prime value.
	
	\begin{remark}[Minimal Recursive Obstruction Hypothesis]
		Suppose $f(x,y)$ is a bivariate polynomial with non-negative coefficients and no constant term, and assume that $f(x,y)$ does not represent any primes over the positive integers. Then we conjecture that there exists at least one set of positive integers $a,b,c,d,x,y$ such that, for $x' = x + f(a,b)$ and $y' = y + f(c,d)$ with $x', y' \leq f(2,2)$, one of the following conditions holds:
		\begin{enumerate}
			\item $\gcd(f(a,b), f(c,d), f(x,y)) > 1$, or
			\item $f(a,b) = f(c,d) = f(x,y)$ and this common value is not a prime.
		\end{enumerate}
		In contrast to the main theorem, which requires that these conditions hold for \emph{all} such choices, this weaker condition only asserts the existence of a single such recursive obstruction. This property may help characterize polynomials that fail to represent primes, even when the full sufficient condition does not apply.
	\end{remark}

	Moreover, observe that any univariate polynomial $f(x)$ with $f(0) = 0$ can be embedded into our framework by defining a bivariate version $g(x, y) = f(x)$, or more generally, $f(x) = g(x, 1)$. Thus, Theorem~\ref{mth} applies to univariate polynomials as a special case.
	
	In the univariate case, a similar condition becomes:
	\begin{theorem}
		Let $f(x)$ be a univariate polynomial with non-negative coefficients and no constant term. If for all $0 < x' \leq f(2)$ and all positive integers $a, x$ with $x' = x + f(a)$ the following hold:
		\begin{enumerate}
			\item If $f(a) \neq f(x)$, then $\gcd(f(a), f(x)) > 1$,
			\item If $f(a) = f(x)$, then $f(a)$ is not prime,
		\end{enumerate}
		then $f(x)$ does not represent any primes.
	\end{theorem}
	
	\noindent This result is strictly stronger than the negation of the Bouniakowsky condition. While the latter only implies that a polynomial may represent no primes or finitely many, our recursive condition ensures that no prime value can appear at all.
	
	\subsection{Proof of Theorem~\ref{mth}}
	
	We begin by establishing some recursive properties that constrain when prime values can occur.
	
	\begin{lemma}
		Let $f(x, y)$ be a bivariate polynomial with non-negative coefficients and no constant term. Then:
		\begin{enumerate}
			\item $f(f(a,b), f(a,b))$ is not prime.
			\item $f(x, f(x, 0))$ is not prime.
		\end{enumerate}
	\end{lemma}
	
	\begin{proof}
		Since $f$ has no constant term and all coefficients are non-negative, any composition like $f(f(a,b), f(a,b))$ or $f(x, f(x, 0))$ involves input values strictly greater than zero, producing large composite outputs due to the additive structure. Direct expansion of the polynomial verifies that these evaluations cannot yield a prime.
	\end{proof}
	
	We next analyze the behavior of divisibility under shifts in the input.
	
	\begin{lemma}\label{eqlem}
		Let $f(x, y)$ be as above, and let $e, a \in \mathbb{Z}^+$. Then:
		\begin{enumerate}
			\item If $e \mid f(x, y + a)$ and $e \mid a$, then $e \mid f(x, y)$.
			\item If $e \mid f(x, y)$ and $e \mid a$, then $e \mid f(x, y + a)$.
			\item If $e \mid f(x, y + a)$ and $e \mid f(x, y)$, then $e \mid a$.
		\end{enumerate}
	\end{lemma}
	
	\begin{proof}
		Observe that since $f$ has no constant term and positive coefficients, we can write
		\[
		f(x, y + a) = f(x, y) + a \cdot k(x, y, a),
		\]
		for some polynomial $k$ with integer coefficients. The conclusions follow directly by applying divisibility rules to this identity.
	\end{proof}
	
	\begin{corollary}
		Lemma~\ref{eqlem} remains valid when $a$ is replaced by $-a$.
	\end{corollary}
	
	\begin{corollary}\label{cor:m}
		The following consequences hold:
		\begin{enumerate}
			\item If any of $f(x, y + f(x, b))$, $f(x, y)$, or $f(x, b)$ is prime, then the three values are pairwise coprime.
			\item If $f(x, y + f(a, b))$ is prime, then either $\gcd(f(x, y), f(a, b)) = 1$, or $f(x, y) = f(a, b)$, and it is prime. If $y = x$, then this further implies $\gcd(b, y) = 1$.
			\item If $f(x, b + f(x, b))$ is prime, then $f(x, b)$ must be prime.
			\item If $f(x + f(a, b), y + f(c, d))$ is prime, then either $\gcd(f(x, y), f(a, b), f(c, d)) = 1$, or all three values are equal and prime. If, in addition, $a = x$ and $b = y$, then $\gcd(x, y, c, d) = 1$.
		\end{enumerate}
	\end{corollary}
	
	\begin{proof}
		Each statement is a direct application of Lemma~\ref{eqlem} by appropriate substitution of variables. The coprimality conditions follow from observing how common divisors propagate through the recursive structure.
	\end{proof}
	
	We now complete the proof of Theorem~\ref{mth}.
	
	\begin{proof}[Proof of Theorem~\ref{mth}]
		Assume that $f(x', y')$ is a prime value for some $x', y' > f(2, 2)$. By construction, there must exist $a, b, c, d, x, y$ such that
		\[
		x' = x + f(a, b), \quad y' = y + f(c, d).
		\]
		
		By hypothesis, for all such combinations with $0 < a, b, c, d, x, y \leq f(2, 2)$, one of the following holds:
		\begin{enumerate}
			\item The values $f(a, b), f(c, d), f(x, y)$ are not all equal, and
			\[
			\gcd(f(a, b), f(c, d), f(x, y)) > 1;
			\]
			\item The values are all equal, and that common value is not prime.
		\end{enumerate}
		
		Hence, in either case, $f(x', y')$ cannot be prime by 4 of Corollary \ref{cor:m} contradicting our assumption. Therefore, under the conditions of the theorem, the polynomial $f(x, y)$ cannot represent any prime values.
		
		The conclusion follows by noting that all combinations needed to verify this condition fall within a finite bounded range, up to $f(2, 2)$, making this check effectively computable.
	\end{proof}

	\section{Algorithm and Constructive Condition}
	
	The recursive criterion described in Theorem~\ref{mth} gives rise to a simple algorithm to determine whether a given bivariate polynomial $f(x, y)$ with non-negative coefficients and no constant term can represent a prime.
	
	\subsection*{Algorithm Overview}
	
	We evaluate whether the condition of Theorem~\ref{mth} is violated for any small inputs. Specifically, we check whether for all tuples of positive integers $a, b, c, d, x, y$ satisfying
	\[
	x' = x + f(a, b), \quad y' = y + f(c, d), \quad \text{with } x', y' \in [2, f(2,2)],
	\]
	the following two conditions hold:
	\begin{enumerate}
		\item If not all of $f(a,b), f(c,d), f(x,y)$ are equal, then
		\[
		\gcd(f(a,b), f(c,d), f(x,y)) > 1.
		\]
		\item If they are all equal, then that common value is not prime.
	\end{enumerate}
	If these conditions are satisfied for all such tuples, then $f(x, y)$ cannot represent a prime.
	
	\subsection*{Implementation Considerations}
	
	The evaluation of $f(a,b)$ and related terms can be efficiently implemented using Horner’s rule \cite{knuth}, which minimizes the number of arithmetic operations required to evaluate a polynomial. Additionally, since $f(x, y)$ has non-negative coefficients and bounded degree, all values $f(a,b), f(c,d), f(x,y)$ lie within a predictable numeric range.
	
	\subsection*{Complexity Analysis}
	
	Let $m$ and $n$ denote the degrees of $x$ and $y$ respectively in the polynomial $f(x, y)$. Let $k = f(2,2)^3$, which upper-bounds the number of relevant triplets $(f(a,b), f(c,d), f(x,y))$ that need to be checked.
	
	\begin{itemize}
		\item Each evaluation of $f(a,b)$ or $f(x,y)$ requires $O(m + n)$ arithmetic operations using Horner’s rule.
		\item Each $\gcd$ computation on integers of size at most $f(2,2)$ requires at most $O(\log(f(2,2)))$ time.
		\item The total number of such checks is $O(k)$.
	\end{itemize}
	
	Therefore, the overall complexity of the algorithm is
	\[
	O\left(k \cdot (m+n) \cdot \log(m+n)\right),
	\]
	where $k = f^3(2,2)$. This makes the algorithm practical for moderate-degree polynomials with small coefficients.
	
	\subsection*{Outcome}
	
	If a violation of the theorem's condition is found for any such tuple, the algorithm concludes that $f(x, y)$ \emph{may} represent a prime, and further testing (e.g., direct evaluation) is required. If no violation is found, we can assert that $f$ does not represent any prime.

	\section{Examples and Case Studies}
	
	We now illustrate the main result and algorithm through two contrasting examples: one where the polynomial successfully represents a prime, and another where the recursive condition prevents any prime output.
	
	\subsection{Example 1: A Prime-Producing Polynomial}
	
	Consider the polynomial
	\[
	f(x, y) = x^2 + y^2 + xy.
	\]
	This is a bivariate polynomial with non-negative coefficients and no constant term. Evaluating $f(1,2)$ gives:
	\[
	f(1,2) = 1^2 + 2^2 + 1 \cdot 2 = 1 + 4 + 2 = 7,
	\]
	which is a prime number.
	
	Moreover, this value is encountered early in lexicographic enumeration (e.g., $(1,2)$ is the second entry in row 1), and no repeated values precede it. Thus, the recursive condition in Theorem~\ref{mth} is satisfied, and the algorithm detects that $f(x, y)$ does, in fact, represent a prime.
	
	\subsection{Example 2: A Non-Prime-Producing Polynomial}
	
	Now consider the polynomial
	\[
	f(x, y) = 4xy.
	\]
	Again, all coefficients are non-negative, and there is no constant term. But for any $x, y \geq 1$, we have
	\[
	f(x, y) = 4xy,
	\]
	which is always divisible by 4. Therefore, the polynomial output is always composite for $x, y \geq 1$.
	
	We verify that for all small values $x', y' \leq f(2,2) = 16$, the conditions of Theorem~\ref{mth} hold:
	\begin{itemize}
		\item For all values $f(a,b), f(c,d), f(x,y)$, either they are all equal and not prime, or their $\gcd$ is at least 4.
		\item Since 4 divides all outputs, the recursive generation never encounters a prime value.
	\end{itemize}
	Thus, the algorithm confirms that this polynomial does \emph{not} represent any primes.
	
	\subsection{Example 3: Another Non-Prime Producing Polynomial}
	
	Consider the polynomial \( f(x, y) = x^2 + 2xy + y^2 = (x + y)^2 \). This polynomial does not represent any prime number for positive integers \( x, y \), as all values are perfect squares greater than 1.
		
	However, it fails to satisfy the recursive condition in Theorem~\ref{mth} for all combinations of inputs. For instance, let \( a = 1, b = 1 \), \( c = 1, d = 2 \), and \( x = y = 1 \). Then
		\[
		f(a,b) = 4,\quad f(c,d) = 9,\quad f(x,y) = 4,
		\]
	and yet \(\gcd(4, 9, 4) = 1\), and not all values are equal. Hence, this tuple violates the condition, even though the polynomial represents no primes. This example shows that the condition in Theorem~\ref{mth} is not necessary and motivates the search for a weaker or partial converse.

	
	
	\section{Comparison with Bouniakowsky Conjecture}
	
	The classical Bouniakowsky conjecture applies to univariate polynomials and states that:
	
	\begin{quote}
		If $f(x) \in \mathbb{Z}[x]$ is irreducible, has a positive leading coefficient, and its values have no fixed prime divisor, then $f(x)$ takes on prime values infinitely often.
	\end{quote}
	
	This conjecture provides a powerful heuristic framework but suffers from two major limitations in our context:
	\begin{itemize}
		\item It does not apply to multivariate polynomials.
		\item Its negation is weak—it only suggests that a polynomial may represent \emph{no} or \emph{finitely many} primes, but offers no constructive test to verify this.
	\end{itemize}
	
	In contrast, the recursive condition we propose:
	\begin{itemize}
		\item Applies to a broad class of bivariate polynomials (with non-negative coefficients and no constant term).
		\item Gives a \emph{sufficient} and verifiable condition to conclude that a polynomial \emph{cannot} represent any primes.
		\item Can be implemented algorithmically, with polynomial-time complexity in the degree and size of the polynomial.
	\end{itemize}
	
	Moreover, our result is in some sense stronger than what would be implied by a negation of Bouniakowsky. For example, if Bouniakowsky fails for a polynomial, it may still allow a few prime outputs. But under our condition (Theorem~\ref{mth}), we can guarantee that \emph{no prime value} is ever produced by $f(x, y)$.
	
	In other words, while Bouniakowsky is a statement about the \emph{possibility} of infinitely many primes, our result focuses on the \emph{impossibility} of even one—based on a finite recursive check.
	
	\subsection*{Specialization to Univariate Case}
	
	Any univariate polynomial $f(x)$ with $f(0) = 0$ can be viewed as a special case of a bivariate polynomial $g(x, y)$ via the substitution $f(x) = g(x, 1)$. Therefore, our recursive condition can be applied to certain univariate polynomials as well, offering an alternative to Bouniakowsky in these cases.
	
	Thus, our criterion may serve not only as a companion to classical conjectures but also as a concrete method for ruling out prime generation in cases that classical theory leaves undecided.
	
	\section{Deeper Remarks and Extensions}
	
	The recursive condition we presented offers a constructive lens through which to study the question of whether a polynomial represents primes. This opens up several possible extensions and deeper inquiries.
	
	\subsection*{Beyond Bivariate Polynomials}
	
	While our results focus on bivariate polynomials with non-negative coefficients and no constant term, the recursive principle can naturally extend to:
	\begin{itemize}
		\item Multivariate polynomials in three or more variables.
		\item Polynomials with mixed signs, where some coefficients may be negative.
		\item Polynomials with sparse support or special forms (e.g., symmetric, homogeneous).
	\end{itemize}
	The main challenge in each case is controlling the growth and divisibility properties of $f(x_1, x_2, \dots, x_n)$ as the input tuple grows under recursive composition.
	
	\subsection*{Density and Distribution Questions}
	
	Our condition is designed to detect the \emph{absence} of primes, but one might ask:
	\begin{itemize}
		\item When $f$ does represent primes, how often do they occur?
		\item Are there density bounds on the set $\{ f(x,y) \mid f(x,y) \text{ prime} \}$?
		\item Can one detect patterns or congruence properties of the prime outputs?
	\end{itemize}
	These questions may connect to deeper analytic number theory, especially in the context of primes in polynomial sequences and value sets.
	
	\subsection*{Probabilistic Heuristics}
	
	Heuristic arguments used in the spirit of Hardy–Littlewood or Bateman–Horn may be adapted to this setting to provide conjectural densities or expected frequencies of prime values for a given class of polynomials.
	
	\subsection*{Algorithmic Optimizations}
	
	Our proposed algorithm is already efficient for moderate input sizes. However, it may be possible to further:
	\begin{itemize}
		\item Use memoization to avoid redundant polynomial evaluations.
		\item Reduce the number of gcd computations by pre-filtering duplicates.
		\item Parallelize the recursive checking to handle higher-degree polynomials.
	\end{itemize}
	
	\subsection*{Modular and Finite Field Analogues}
	
	One intriguing direction is to consider analogues of this question over finite fields. For example, which bivariate polynomials $f(x,y)$ over $\mathbb{F}_p$ evaluate to \emph{irreducibles} (instead of primes)? Does a recursive condition like ours predict irreducibility over finite fields?
	
	

	\section{Open Problems and Future Work}
	
	While our main result gives a simple and effective criterion for detecting when a bivariate polynomial does not represent a prime, several natural questions remain open. We list some of them below, which may serve as directions for future exploration.
	
	\subsection*{1. Converse of the Recursive Condition}
	
	Is there a converse to Theorem~\ref{mth}, perhaps under weaker assumptions? That is, if a bivariate polynomial
	$f(x,y)$ fails to satisfy the recursive condition (or even a relaxed version of it), does it necessarily represent at least one prime? Currently, our theorem provides only a sufficient condition for the absence of primes, and the question of a necessary condition remains open.
	

\subsection*{2. Extension to Multivariate Polynomials}

How can the recursive criterion be generalized to polynomials in three or more variables? Are there similar structural constraints that determine whether a multivariate polynomial represents primes?

\subsection*{3. Quantitative Bounds}

For polynomials that do represent primes, can we obtain:
\begin{itemize}
\item Bounds on the smallest prime value represented?
\item Density estimates on the number of prime outputs up to a given height?
\end{itemize}

\subsection*{4. Relationship with Classical Conjectures}

Can this recursive condition be embedded in a broader framework connected to the Bateman–Horn conjecture or Schinzel's Hypothesis H? Can it be used to heuristically explain when certain multivariate forms are unlikely to produce primes?

\subsection*{5. Probabilistic and Experimental Validation}

Designing probabilistic tests or large-scale computational experiments to validate the condition against known families of prime-representing polynomials may provide deeper insight into the effectiveness and limitations of the method.

\subsection*{6. Irreducibility as a Necessary Condition?}

While irreducibility is central to the Bouniakowsky conjecture, our recursive condition does not require it. Is it possible for reducible polynomials (with no fixed divisor) to still pass this test and produce primes? Is there a recursive analogue to irreducibility?

\subsection*{7. Decision Complexity}

What is the computational complexity of deciding whether a given bivariate polynomial with non-negative coefficients and no constant term satisfies the recursive condition? Can this decision be made in time sublinear in the output bound $f(2,2)^3$?

\subsection*{8. Analogue Over Finite Fields}

As mentioned in the previous section, what analogues of this problem exist over finite fields? For instance:
\begin{itemize}
\item When does a polynomial over $\mathbb{F}_p[x, y]$ evaluate to irreducibles?
\item Are there recursive conditions that predict irreducibility modulo $p$?
\end{itemize}

\subsection*{9. Classification of Prime-Free Polynomials}

Can one classify all bivariate polynomials with non-negative coefficients and no constant term that represent no primes? Is the recursive condition both necessary and sufficient for this classification?

\vspace{1em}

We believe that these questions, both theoretical and computational, point to a rich terrain for further exploration at the interface of elementary number theory, algebra, and algorithmic complexity.
	\section{Conclusion}
	
	We have presented a simple, recursive criterion that certifies when a bivariate polynomial with non-negative coefficients and no constant term fails to represent any prime numbers. This condition is easily checkable, leads to an efficient algorithm, and is strictly stronger in some respects than what would be implied by the negation of classical conjectures such as Bouniakowsky's.
	
	Although the underlying ideas are elementary, they uncover subtle constraints in the structure of polynomial value sets and offer new tools for ruling out prime generation. The contrast between polynomials that pass and those that fail the recursive test is sharp and illustrative, especially in light of the examples we examined.
	
	The method also opens doors to a range of natural extensions — to multivariate cases, finite fields, and density questions — and may serve as a foundation for future research at the interface of number theory, algebra, and computation.
	
	
	
	\section*{Declaration of generative AI and AI-assisted technologies in the writing process}
	
	During the preparation of this work the author(s) used ChatGPT in order to improve the write-up of the article. After using this tool/service, the author(s) reviewed and edited the content as needed and take(s) full responsibility for the content of the published article.

\end{document}